\documentclass[reqno, 11pt]{amsart}

\usepackage[toc]{appendix}
\usepackage{amsthm} 				
\usepackage{amsmath}				
\usepackage{amssymb}				
\usepackage{amsfonts}
\usepackage{hyperref}
\usepackage{csquotes}				
\usepackage[T1]{fontenc} 				

\usepackage[a4paper]{geometry}				

\usepackage{tikz-cd}					
\usetikzlibrary{arrows}

\usepackage[all]{xy}					
\usepackage{pgf,tikz}
\usetikzlibrary{calc}
\usepackage{mathrsfs}
\usepackage{amsfonts}
\usetikzlibrary{arrows}
\usepackage[mathscr]{eucal}
\usepackage[normalem]{ulem}
\usepackage{enumitem}

\usepackage{graphicx}
\usepackage{caption}
\usepackage{subcaption}

\newtheorem{theorem}{Theorem}[section]
\newtheorem{corollary}[theorem]{Corollary}
\newtheorem{proposition}[theorem]{Proposition}
\newtheorem{lemma}[theorem]{Lemma}

\newtheorem*{theorem*}{Theorem}

\theoremstyle{definition}
\newtheorem{definition}[theorem]{Definition}
\newtheorem{example}[theorem]{Example}
\newtheorem{remark}[theorem]{Remark}

\newcommand{\Hom}{\mbox{Hom}}

\newcommand{\F}{\mathcal{F}}
\newcommand{\E}{\mathcal{E}}

\newcommand{\M}{\mathbb{M}}
\renewcommand{\P}{\mathscr{P}}
\newcommand{\C}{\mathcal{D}}

\newcommand{\soc}{\mbox{soc}}

\begin{document}

\title{Characterisations of trivial extensions}

\author[Fernandez]{Elsa Fernandez}
\address{\textbf{[EF]} 
Facultad de Ingenier\'ia, Universidad Nacional de la Patagonia San Juan Bosco. (9120) Puerto Madryn, Argentina.}
\email{elsafer9@gmail.com}

\author[Schroll]{Sibylle Schroll}
\address{\textbf{[SS]}
Mathematical Institut
University of Cologne
Weyertal 86-90
(50931) Köln
Germany and Institutt for matematiske fag, NTNU, N-7491 Trondheim, Norway}
\email{schroll@math.uni-koeln.de}

\author[Treffinger]{Hipolito Treffinger}
\address{\textbf{[HT]} Institut Math\'ematique Jussieu - Paris Rive Gauche
Univerist\'e Paris Cit\'e, Bâtiment Sophie Germain
5, rue Thomas Mann.
(75013) Paris, 
France.}
\email{treffinger@imj-prg.fr}

\author[Trepode]{Sonia Trepode}
\address{\textbf{[ST]}
Centro Marplatense de Investigaciones Matemáticas. FCEyN, Universidad Nacional de
Mar del Plata, CONICET. Funes 3350. (7600) Mar del Plata, Argentina.}
\email{strepode@mdp.edu.ar}

\author[Valdivieso]{Yadira Valdivieso}
\address{\textbf{[YV]}
Departamento de Actuaría, Física y Matemáticas, Universidad de las Américas Puebla, Ex Hacienda Sta. Catarina Mártir S/N, San Andrés Cholula, 72810 Puebla, México}
\email{yadira.valdivieso@udlap.mx}

\maketitle

\begin{abstract}
In this paper we give a characterisation of trivial extension algebras in terms of quivers with relations. 
This result is based on a explicit description of the ideal of relations of the trivial extension of an algebra, given by the first author in the appendix.
We also give a new proof of Wakamatsu's theorem in terms of their quiver and relations, which determines when two given algebras have isomorphic trivial extensions.
\end{abstract}

\section{Introduction}

Split extensions of a ring by a bimodule are classical constructions that appear in many different contexts.
For example, Hochschild \cite{Hochschild1945} showed that the trivial extension of a ring $R$ by an $R$-$R$-bimodule $M$ corresponds to the zero element in the second cohomology group $H^2(R,M)$. 
In commutative algebra, Nagata \cite{Nagata} used split-extensions to show that any module over a commutative ring can be thought of as an ideal. 
More recently, split extensions have played a central role in the study of cluster-tilted algebras \cite{ABScluster-tilted, BFPPT10}, in the connections between gentle algebras and symmetric special biserial algebras  \cite{Green, Schroll2015}, in higher homological algebra \cite{Grant2019} and the Hochschild (co)homology of split-extensions has been studied, for example, in \cite{Assem2013, Assem2016, Bergh2017}. 

Although the theory of trivial extensions has been greatly developed, in general, given an explicit algebra in terms of  quiver and relations, there has not been an explicit construction of the structure of the trivial extension as an algebra, except for certain cases \cite{Fernandez2002, FernandezTesis, HernandezTesis}.
Note that by the trivial extension of an algebra $A$, we mean the algebra $T(A) = A \ltimes D(A)$ where $D(A)$ is $A$-$A$-bimodule given by the dual of $A$.  
In the appendix of this paper, an explicit construction of the trivial extension is developed by the first author (see Theorem~\ref{thm:1.1}).  
Namely given an algebra $A = KQ_A/I_A$ the ideal of relations  of the  trivial extension $T(A)$ is explicitly constructed, thus completing the description of $T(A)$ since in \cite{Fernandez2002} the quiver of $T(A)$ was already constructed. 

Furthermore, it has been an open question to know which algebras are isomorphic to the trivial extension of a finite dimensional algebra. 
Since trivial extensions are symmetric, the question reduces to asking when is a symmetric algebra isomorphic to a trivial extension. We address this question for all finite dimensional symmetric $K$-algebras, see Theorem~\ref{thm:1.2}. 
Finally, we turn to the question of when two algebras have isomorphic trivial extensions, a question already considered in the abstrac setting by Wakamatsu in \cite{Wakamatsu1984}. 
Lastly, in Theorem~\ref{thm:1.3}, we give a precise characterisation of the algebras that have isomorphic trivial extensions in the language of admissible cuts which were introduced in \cite{Fernandez2006}.

We now  state the main results of this paper. 
For this we recall that in \cite{Fernandez2002} the quiver $Q_{T(A)}$ of the trivial extension $T(A)$ of a finite dimensional algebra $A =KQ_A/I_A$ has been described. 
Namely, the set of vertices $(Q_A)_0$ and $(Q_{T(A)})_0$ coincide.  Then given a $K$-basis $\M=\{p_1, \dots, p_n\}$ of $\soc_{A^e}A$,
the set of arrows $(Q_{T(A)})_1$ is the disjoint union of $(Q_A)_1$ and $\{\beta_{p_1}, \dots, \beta_{p_n}\}$ where $s(\beta_{p_i})=t(p_i)$ and $t(\beta_{p_i})=s(p_i)$ for every $p_i \in \M$.
Given a path $p$, we call a path $q$ such that $pq$ is a cycle a supplement of $p$ (see Definition~\ref{def:supplement}. For the notion of elementary cycle, we refer to Definition~\ref{def:elementary}. 
In the appendix of this paper the following is shown. 

\begin{theorem}[Theorem~\ref{thm:ideal}]\label{thm:1.1}
Let $A= KQ_A / I_A$ be a finite-dimensional algebra and let $T(A) = KQ_{T(A)}/I_{T(A)}$ be its trivial extension.
Then the quiver $Q_{T(A)}$ is as above and the ideal $I_{T(A)}$ is generated by the union of the following sets.

\begin{enumerate}
    \item A generating set of the ideal of relations $I_A$ of $A$.
    \item The paths that are not contained in an elementary cycle.
    \item For any vertices $x$ and $y$ in $Q_{T(A)}$, the linear combinations of paths $\rho\in e_xKQ_{T(A)}e_y$ such that $q \rho\in I_x'$ or $\rho q\in I_{y}'$ for any supplement path $q$ in an elementary cycle $C$.
\end{enumerate}
\end{theorem}

Having this complete description of the ideal of relations of a trivial extension, we are able to give a complete characterisation of when a symmetric algebra is isomorphic to  a trivial extension of some finite-dimensional $K$-algebra. 

\begin{theorem}[Theorem~\ref{thm:characterisation_trivial_extensions}]\label{thm:1.2}
Let $A=KQ/I$ be an algebra. 
Then $A$ is isomorphic to the trivial extension of some finite-dimensional $K$-algebra if and only if:

\begin{enumerate}
    \item[(a)] There is a presentation of $A$ for which there exists a set $\E$ of distinguished cycles in $Q_A$ with weight function $\omega: \E \to K$ and
    \item[(b)] There is an allowable cut $\C =\{\gamma_1, \cdots, \gamma_t \}$ of $A$ such that the quotient $B=A/\langle \C\rangle$ verifies the following:
\begin{enumerate}
\item [(i)] $A$ is a split-by-nilpotent extension of $B$.
\item [(ii)] The supplements of the cut arrows in the cycles in $\mathcal{E}$ are in one-to-one correspondence with the elements of a basis of $\soc_{B^e} B$.
\end{enumerate}
\end{enumerate}
In this case the two-sided ideal $\langle \C \rangle$ is isomorphic to $DB$ as a $B$-$B$-bimodule and $A$ is isomorphic to $T(B)$.
\end{theorem}

We note that the previous result reduces the problem of identifying if an algebra is a trivial extension to the problem of determining whether an algebra is a split-by-nilpotent extension, a problem that has been solved in \cite{ACT}. Thus our result gives a complete answer which in combination with \cite{ACT} should be implementable in terms of a computer algorithm.

\medskip

In the 70's,  M\"uller \cite{MU74}, Green and Reiten \cite{GR76}, Iwanaga and Wakamatsu \cite{IW79}, and Hughes and Waschb\"usch \cite{HW83} studied the relationship between the representation type of the algebra $T(A)$ and the algebra $A$. 
In particular, in \cite{HW83}, they gave a complete description of the representation-finite trivial extension of algebras which relies on whether two trivial extensions are isomorphic or not. 
This result motivated Wakamatsu to study when two trivial extensions are isomorphic, see \cite{Wakamatsu1984}. 
He gave, for two Artin algebras, necessary and sufficient conditions for having isomorphic trivial extensions.
In more precise terms, he shows that two algebras $A$ and $A'$ have isomorphic trivial extensions if and only if they are split-by-nilpotent extensions of a common subalgebra $S$ by a $S$-$S$-bimodule $M$ and its dual $D(M)$, respectively.

The original motivation for this paper was to give an explicit description of the relationship of two algebras $A$ and $A'$ that have isomorphic trivial extensions.
This is achieved in our next result, which gives an explicit description of the algebra $S$ and the extending $S$-$S$-bimodules in the case where $A$ and $A'$ are given by quiver and relations. We also  give an explicit characterisation of when the trivial extensions of two algebras are isomorphic. 
Our result is shown by a proof which is independent of Wakamatsu's proof. 

\begin{theorem}[Theorem~\ref{thm:Wakamatsu_for_path_algebras}]\label{thm:1.3}
Let $A= KQ_A/I_A$ be a finite-dimensional algebra with trivial extension $T(A)$ and set $(Q_A)_1=\{\alpha_1, \ldots, \alpha_n\}$ and $\{\beta_1, \dots, \beta_t\}= (Q_{T(A)})_1 \setminus (Q_A)_1$.
The following are equivalent
\begin{enumerate}[label=(\alph*)]
\item $T(A)\cong T(A')$.
\item There exists an admisible cut $\C$ of the form
$$\C\;=\;\{ \alpha_1, \ldots, \alpha_r, \beta_1, \ldots, \beta_s: \alpha_i \in (Q_A)_1 \text{ and } \beta_j\not \in (Q_A)_1\}$$ with $A'\cong T(A)/ \langle \C \rangle$, $T(A)$ is a split-by-nilpotent extension of $A'$ and the supplements in the elementary cycles of the cut arrows are in  one-to-one correspondence with the elements of  $\soc_{A'^e}
A'$.
\item $A \cong S \ltimes M$ and $A' = S \ltimes D(M)$, where 
\begin{enumerate}
    \item[(i)] $S$ is the subalgebra of $T(A)$ generated by $\sum_{x \in (Q_{T_A})_0} e_x$ the identity of $T(A)$ and $\alpha_{r+1}, \ldots, \alpha_n$,
    \item[(ii)] $M = S \left\langle \alpha_1, \ldots, \alpha_r \right\rangle S$,
    \item[(iii)] and $D(M) = S \left\langle \beta_{s+1}, \ldots, \beta_t\right\rangle S$.
\end{enumerate}
\end{enumerate}
\end{theorem}

\section{Background}\label{sec:background}

In this paper, by an algebra we mean a basic connected finite-dimensional algebra over an algebraically closed field $K$.
It is well-known that every such algebra $A$ is Morita equivalent to the path algebra of a (finite) quiver $Q$ modulo an admissible ideal $I$ of
$KQ$.
For the purposes of this paper, we always identify $A$ with $KQ_A/I_A$.
If $A=KQ_A/I_A$, we say that $Q_A$ is the quiver of $A$.
In addition, every element of a set of generators of $I_A$ is called a \textit{relation} of $A$ and we refer to $I_A$ as the ideal of relations of $A$.

Given a quiver $Q$, we denote by $Q_0$ the set of vertices of $Q$ and $Q_1$ the set of arrows of $Q$.
For every arrow $\alpha \in Q_1$, $s(\alpha)$ and $t(\alpha)$ denote the source and the target of $\alpha$, respectively.
A \textit{path} $p$ of length $t$ from the vertex $x$ to the vertex $y$ in $Q$ is an ordered set of arrows $\{\alpha_1, \dots, \alpha_t\}$ such that $s(\alpha_1)=x$, $t(\alpha_t)=y$ and $t(\alpha_i)=s(\alpha_{i+1})$ for all $1 \leq i \leq t-1$.
By abuse of notation we write $p = \alpha_1\alpha_2 \dots \alpha_t$ and we define $s(p):=s(\alpha_1)$ and $t(p):=t(\alpha_t)$.
Also, for every vertex $x \in 
(Q_A)_0$ we denote by $e_x$ the stationary path at vertex $x$, that is  the unique path of length $0$ from $x$ to $x$.
We say that a path $p$ is a \textit{cycle} if $p$ is of  strictly positive length and $s(p)=t(p)$.
A cycle $p$ is said to be \textit{non-zero} if $p \not \in I_A$.

Let $A$ be an algebra and $M$ be a $A$-$A$-bimodule equipped with a multiplication map $\mu: M\otimes_A M \to M$.
Then the $K$-vector space $B= A\oplus M$ has  the structure of an algebra with the following multiplication
$$(a_1, m_1)(a_2, m_2)= (a_1a_2, a_1m_2 + m_1a_2 + \mu(m_1\otimes m_2)),$$
where $a_1, a_2 \in A$ and $m_1, m_2 \in M$.
If the map $\sigma: A \to B$ defined by $\sigma(a)= (a,0)$ is a morphism of algebras we say that $B$ is a \textit{split extension of $A$ by $M$}.
Moreover, if $M$ is nilpotent for $\mu$ we say that $B$ is a \textit{split-by-nilpotent extension of $A$ by $M$} and we denote it as $B= A*M$.
If $\mu(m_1, m_2)=0$ for all $m_1, m_2 \in M$ we say that $B$ is the \textit{trivial extension of $A$ by $M$} and we denote it by $A\ltimes M$.
Finally, if $M=D(A)$ with its natural bimodule structure, where $D=\Hom_A(-, K)$, we say that $B=A\ltimes D(A)$ is simply the \textit{trivial extension} of $A$ and we denote it
by $T(A)$.

We recall from \cite{Fernandez2002} an explicit description of the ordinary quiver of
$T(A)$ based on the quiver of $A$.  This description depends on a $K$-basis $\M=\{p_1, \dots, p_n\}$ of $\soc_{A^e}A$. 
The quiver $Q_{T(A)}$ is then defined as follows.
The set of vertices $(Q_A)_0$ and $(Q_{T(A)})_0$ coincide.
The set of arrows $(Q_{T(A)})_1$ is the disjoint union of $(Q_A)_1$ and $\{\beta_{p_1}, \dots, \beta_{p_n}\}$ where $s(\beta_{p_i})=t(p_i)$ and $t(\beta_{p_i})=s(p_i)$ for every $p_i \in \M$.
A set of generators of the ideal of relations $I_{T(A)}$ of $T(A)$ is given in Theorem~\ref{thm:ideal}.

Finally, we recall the notion of admissible cut and elementary cycles, introduced in \cite{Fernandez2002} to study when two algebras of a certain family have isomorphic trivial extension.

\begin{definition}\label{def:elementary}
Let $\M=\{p_1, \dots, p_n\}$ be a basis of $\soc _{A^e} A$. An oriented cycle $C=\alpha_1\dots\alpha_t\beta_{p_j}$ in $KQ_{T(A)}$ is said to be \textit{elementary} if $\alpha_1
\dots \alpha_n$ is a path in $KQ_A$, $p_j\in \M$ and $p_j^*(\alpha_1\dots\alpha_t)\neq 0$, where $p_j^*$ is the dual of $p_j$.
In this case, we say that $\omega(C)=p_j^*(\alpha_1\dots\alpha_t)\in K$ is the \textit{weight} of $C$.
\end{definition}

\begin{definition}\label{def:supplement}
Let $Q$ be a quiver, $C$ a cycle of $Q$, and $q$ a path contained in $C$. The \emph{supplement}  of $q$ in $C$ is the path $p$ such that $C = qp$ up to rotation. Note that if $C =q$ then $p$ is the stationary path $e_{(s(q))}$.

\end{definition}

\begin{definition}\label{def:admissible cut}
Let $A= KQ/I$ be a finite-dimensional algebra. Then an \emph{admissible cut $\C$ of $T(A)$} is a subset of arrows of $Q_{T(A)}$   containing  exactly one arrow in each
elementary cycle of $T(A)$ in such a way that if an arrow in $\C$ appears in an elementary cycle $C$ then it appears only once in $C$. 
\end{definition}

\begin{remark}\label{rmk:trivial_cut}
The set of new arrows $\C=\{\beta_{p_1}, \dots, \beta_{p_j}\}$ added in the construction of the quiver $Q_{T(A)}$ form an admissible cut of $T(A)$.
Moreover, it is easy to see that $A$ is isomorphic to the quotient $T(A)/\langle \C \rangle$ where $\langle \C \rangle$ is the two-sided of $T(A)$ ideal generated by $\C$.
\end{remark}

We now give three examples. 
In the first example, we show that the weights of the elementary cycles are not always equal to one, in the second we show that not every arrow in an elementary cycle is included in an admissible cut.
In the third example, we show that not every admissible cut $\C$ of $T(A)$ gives rise to an algebra $B = T(A) / \langle \C \rangle $ such that $T(A)$ is a split-by-nilpotent extension of $B$.

\begin{example}
Consider the algebra $A=KQ_A/I_A$ with quiver 
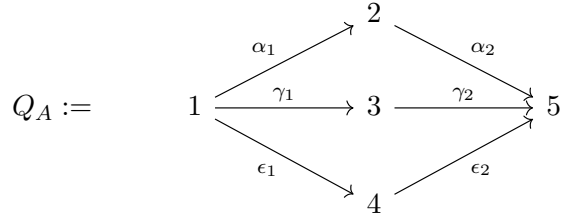
\[\begin{tikzcd}
	&&& 2 \\
	{Q_A:=} & 1 && 3 && 5 \\
	&&& 4
	\arrow["{\alpha_1}", from=2-2, to=1-4]
	\arrow["{\gamma_1}", from=2-2, to=2-4]
	\arrow["{\epsilon_1}"', from=2-2, to=3-4]
	\arrow["{\gamma_2}", from=2-4, to=2-6]
	\arrow["{\alpha_2}", from=1-4, to=2-6]
	\arrow["{\epsilon_2}"', from=3-4, to=2-6]
\end{tikzcd}\]
quotiented by the ideal $I_A= \langle \lambda_1 \alpha_1 \alpha_2 + \lambda_2 \gamma_1 \gamma_2 + \lambda_3 \epsilon_1 \epsilon_2 \rangle$. 
Then  soc$_{A^e} A$  is two dimensional with basis $\{p_1= \gamma_1 \gamma_2, p_2=\epsilon_1 \epsilon_2\}$.
Then $T(A)= KQ_{T(A)}/ I_{T(A)}$ where 
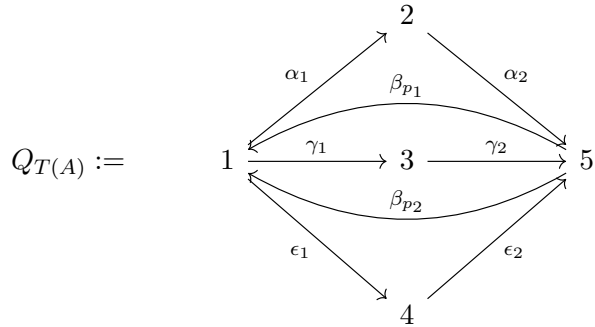
\[\begin{tikzcd}
	&&& 2 \\
	\\
	{Q_{T(A)}:=} & 1 && 3 && 5 \\
	&&  \\
	&&& 4
	\arrow["{\alpha_1}", from=3-2, to=1-4]
	\arrow["{\gamma_1}", from=3-2, to=3-4]
	\arrow["{\epsilon_1}"', from=3-2, to=5-4]
	\arrow["{\gamma_2}", from=3-4, to=3-6]
	\arrow["{\alpha_2}", from=1-4, to=3-6]
	\arrow["{\epsilon_2}"', from=5-4, to=3-6]
	\arrow["{\beta_{p_1}}"', bend right, from=3-6, to=3-2]
	\arrow["{\beta_{p_2}}"', bend left, from=3-6, to=3-2]
\end{tikzcd}\]
In this case we have four different elementary cycles, $C_1 = \beta_{p_1}\gamma_1\gamma_2$, $C_2= \beta_{p_2}\epsilon_1\epsilon_2$, $C_3 = \beta_{p_1}\alpha_1\alpha_2$ and $C_4= \beta_{p_2}\alpha_1\alpha_2$.
We calculate 
$$w(C_1)= p_1^*(\gamma_1\gamma_2)=1$$
$$w(C_2)= p_2^*(\epsilon_1\epsilon_2)=1$$
$$w(C_3)= p_1^*(\alpha_1\alpha_2)=p_1^*(-\lambda_2 \gamma_1\gamma_2 - \lambda_3 \epsilon_1\epsilon_2) = -\frac{\lambda_2}{\lambda_1}$$
$$w(C_4)= p_2^*(\alpha_1\alpha_2)=p_2^*(-\lambda_2 \gamma_1\gamma_2 - \lambda_3 \epsilon_1\epsilon_2) = -\frac{\lambda_3}{\lambda_1}.$$
\end{example}

\begin{example}\label{ex:nonelementarycycle}
Let $Q$ be the quiver
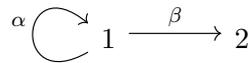
\[
\begin{tikzcd}[column sep= 35pt]
1 \arrow[out=210,in=140,loop, "\alpha"] \ar[r, "\beta", shift left]& 2
\end{tikzcd}\]
and $I$ the ideal generated by $\alpha^3$. Denote by $A$ the algebra $KQ/I$.
Then the set $\M=\{\alpha^2\beta\}$ is a $K$-basis for $\soc_{A^e}A$ and $Q_{T(A)}$ is the following quiver
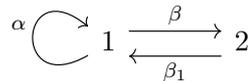
\[
\begin{tikzcd}[column sep= 35pt]
1 \arrow[out=210,in=140,loop, "\alpha"] \ar[r, "\beta", shift left=1ex]& 2 \ar[l,shift left=1ex, "\beta_{1}"]
\end{tikzcd}\]

Then $C=\beta_1\alpha^2\beta$ is the only elementary cycle in $KQ_{T(A)}$, up to cyclic permutations. It follows from Definition~\ref{def:admissible cut} that the only admissible cuts are $\{\beta\}$ and $\{\beta_1\}$.

Observe that if we consider the set $\{\alpha\}$ as a cut, then $A/\langle \alpha \rangle$ is
isomorphic to the algebra $B=KQ'/I'$ where $Q'$ is the quiver
\[
\begin{tikzcd}
1 \ar[r, "\beta", shift left=1ex]& 2 \ar[l,shift left=1ex, "\beta_{1}"]
\end{tikzcd}\]
and $I'=(\beta_1\beta)$, which is an algebra of dimension 6 and therefore $T(B)$ is an algebra of dimension 12, while $T(A)$ is a 14-dimensional algebra. Hence $T(A)$ is not isomorphic to $T(B)$.

We note that in this example there are nonzero cycles in $T(A)$ which are not elementary cycles, namely $\alpha$, $\alpha^2$ and $\beta\beta_1$.
\end{example}

\begin{example}
Let $A$ be an algebra given by $Q$
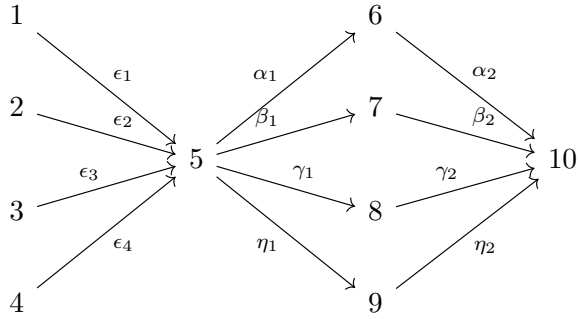
\begin{center}
\begin{tikzcd}
1 \ar{ddrr}{\epsilon_1} & &  && 6 \ar{rrdd}{\alpha_2} && \\
2 \ar{drr}{\epsilon_2} & & && 7 \ar{rrd}{\beta_2} && \\  [-15pt]
  & & 5 \ar{uurr}{\alpha_1}\ar{urr}{\beta_1}\ar{drr}{\gamma_1}\ar[swap]{ddrr}{\eta_1}  && && 10\\  [-15pt]
3 \ar{urr}{\epsilon_3} & & && 8 \ar{urr}{\gamma_2} &&\\
4 \ar[swap]{uurr}{\epsilon_4}& & && 9 \ar[swap]{uurr}{\eta_2}&&
\end{tikzcd}
\end{center}
and the relations $$\alpha_1\alpha_2 +\beta_1\beta_2 +\gamma_1\gamma_2+ \eta_1\eta_2=0$$
$$\epsilon_1\gamma_1\gamma_2=0, \, \epsilon_1\eta_1\eta_2=0, \,\epsilon_2\gamma_1\gamma_2=0, \, \epsilon_2\eta_1\eta_2=0,$$ $$\epsilon_3\alpha_1\alpha_2=0, \,
\epsilon_3\beta_1\beta_2=0, \, \epsilon_4\alpha_1\alpha_2=0,\, \epsilon_4\beta_1\beta_2=0 .$$

Then the following elements are a $K$-basis for $\soc_{A^e}A$.

\begin{eqnarray*}
&\epsilon_1\alpha_1\alpha_2= -\epsilon_1\beta_1\beta_2, \, \epsilon_2\alpha_1\alpha_2=-\epsilon_2\beta_1\beta_2,\, \epsilon_4\gamma_1\gamma_2=-\epsilon_4\eta_1\eta_2\\
& \epsilon_3\gamma_1\gamma_2=-\epsilon_3\eta_1\eta_2,\, \alpha_1\alpha_2-\beta_1\beta_2,\,
\gamma_1\gamma_2-\eta_1\eta_2, \\
& \epsilon_1\gamma_1,\, \epsilon_1\eta_1,\,
\epsilon_2\gamma_1, \, \epsilon_2\eta_1, \, \epsilon_3\alpha_1, \, \epsilon_3\beta_1, \, \epsilon_4\alpha_1, \textrm{\, and\, } \epsilon_4\beta_1
\end{eqnarray*}

The quiver $Q_{T(A)}$ is the following.

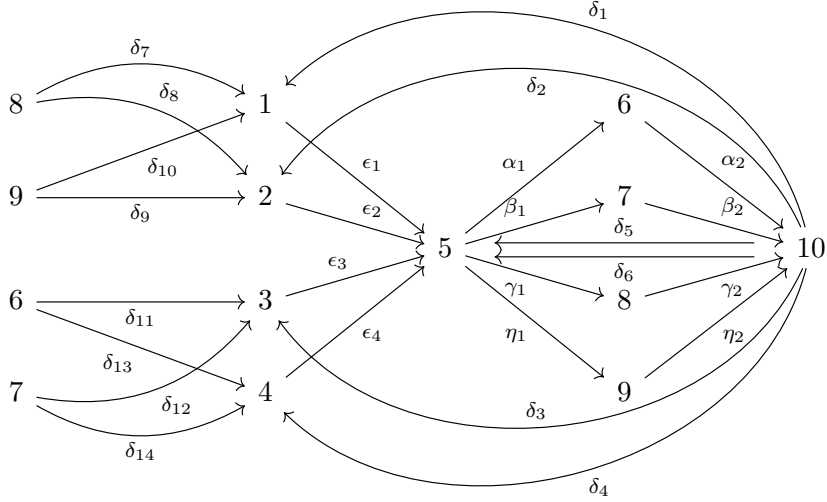
\begin{center}
\begin{tikzcd}
8 \ar[bend left]{rrr}{\delta_{7}} \ar[bend left]{rrrd}{\delta_{8}} && & 1 \ar{ddrr}{\epsilon_1} & &  && 6 \ar{rrdd}{\alpha_2}  && \\
9 \ar[swap]{rrr}{\delta_{9}} \ar[swap]{rrru}{\delta_{10}}&& & 2 \ar{drr}{\epsilon_2} & & && 7 \ar{rrd}{\beta_2}  && \\  [-15pt]
 & & & & & 5 \ar{uurr}{\alpha_1}\ar{urr}{\beta_1}\ar[swap]{drr}{\gamma_1}\ar[swap]{ddrr}{\eta_1}  && && 10  \ar[bend left=60, swap]{lllllld}{\delta_3} \ar[bend
 right=60,]{llllllu}{\delta_2} \ar[bend right=60, swap]{lllllluu}{\delta_1} \ar[bend left=60]{lllllldd}{\delta_4} \ar[shift right, swap, shorten <= 1em, shorten >=
 1em]{llll}{\delta_5} \ar[shift left, shorten <= 1em, shorten >= 1em]{llll}{\delta_6} \\  [-15pt]
6 \ar[swap]{rrr}{\delta_{11}} \ar[swap]{rrrd}{\delta_{13}}& && 3 \ar{urr}{\epsilon_3} & & && 8 \ar[swap]{urr}{\gamma_2}  &&\\
7 \ar[bend right, swap]{rrru}{\delta_{12}} \ar[bend right, swap]{rrr}{\delta_{14}} & &&4 \ar[swap]{uurr}{\epsilon_4}& & && 9 \ar[swap]{uurr}{\eta_2} &&
\end{tikzcd}
\end{center}
Where the vertices labelled by $6,7,8,9$ should be identified.

There are 18 elementary cycles, which are listed below.
\begin{equation*}
\begin{array}{llll}
C_1=\epsilon_1\alpha_1\alpha_2\delta_1, & C_2=\epsilon_1\beta_1\beta_2\delta_1, & C_3=\epsilon_2\beta_1\beta_2\delta_2, & C_4=\epsilon_2\alpha_1\alpha_2\delta_2,\\
C_5= \epsilon_3\gamma_1\gamma_2\delta_3, & C_6=\epsilon_3\eta_1\eta_2\delta_3, &C_7=\epsilon_4\gamma_1\gamma_2\delta_4, & C_8=\epsilon_4\eta_1\eta_2\delta_4,\\
C_9=\alpha_1\alpha_2\delta_5, & C_{10}= \gamma_1\gamma_2\delta_6, & C_{11}=\epsilon_1\gamma_1\delta_7, & C_{12}=\epsilon_2\gamma_1\delta_8\\
C_{13}=\epsilon_1\eta_1\delta_9, & C_{14}=\epsilon_2\eta_1\delta_{10}, &
C_{15}= \epsilon_3\alpha_1\delta_{11},& C_{16}=\epsilon_3\beta_1\delta_{12}\\
C_{17}=\epsilon_4\alpha_1\delta_{13},& C_{18}=\epsilon_4\beta_{1}\delta_{14}& &
\end{array}
\end{equation*}

Observe that the set $\C=\{\alpha_1, \beta_1, \delta_3, \delta_4, \delta_6,\delta_7, \delta_8, \delta_9, \delta_{10} \}$ is an admissible cut of $T(A)$, however $B=T(A)/\langle
\C \rangle$ is not a subalgebra of $T(A)$ since $\gamma_1\gamma_2+ \eta_1\eta_2=0$ is the induced relation from the minimal relation
$\alpha_1\alpha_2+\beta_1\beta_2+\gamma_1\gamma_2+\eta_1\eta_2=0$ in $A$ and therefore $T(A)$ is not a split-by-nilpotent extension of $B$ by \cite[Theorem 2.5]{ACT}.
\end{example}

\section{Characterisation of trivial extensions}
\label{sec:characterisation}
In this section we give a complete characterisation of the algebras that are isomorphic to the trivial extension of some finite dimensional $K$-algebra.

Before proving the main result of this section, Theorem~\ref{thm:characterisation_trivial_extensions}, we introduce the necessary terminology and prove some preliminary results.

\begin{definition}\label{allowablecut2}
Let $A=KQ/I$ be an algebra and let $\mathcal E = \{C_1, \dots, C_n\}$ be a set of non-zero cycles of length at least two. 
Then an \emph{allowable cut $\C$ of $A$ with respect to $\mathcal{E}$} is a subset of arrows of $Q_{A}$ containing  exactly one arrow in each cycle of $\mathcal{E}$ in such a way that if an arrow in $\C$ appears in a cycle $C$ of $\mathcal{E}$ then it appears only once in $C$. 
\end{definition}

\begin{remark}
If $A$ is a trivial extension of an algebra $A'$, then any admissible cut $\C$ of $A$ is an allowable cut with respect the set of elementary cycles.
\end{remark}

In what follows, whenever the set $\mathcal{E}$ of cycles is clear from the context, we will simply say that $\C$ is an allowable cut of $A$, omitting the reference to $\mathcal{E}$.

\begin{definition}
Let $A=KQ_A/I_A$ be an algebra and consider a set $\mathcal{E}$ of non-zero cycles of length at least two and a weight function $\omega: \mathcal{E} \to K$. 
Then for every vertex $x \in (Q_A)_0$ we define $I_x^{\mathcal{E}}$ to be the two-sided ideal in $KQ$ generated by
\begin{enumerate}[label=(\roman*)]
\item Oriented cycles starting and ending at $x$ that do not belong to $\mathcal E$.
\item The elements of the form $\omega (C')C-\omega (C)C'$, where $C, C'\in \mathcal E$.
\end{enumerate}
\end{definition}

\begin{definition}\label{def:distinguished}
Let $A=KQ/I$ be an algebra. 
A finite set $\mathcal{E}$ of non-zero cycles of length at least two is said to be a set of \textit{distinguished cycles} if it admits an allowable cut $\C =\{\alpha_1, \cdots \alpha_r \}$ of $A$ and a weight function $\omega : \mathcal{E} \to K$ such that the following holds.
\begin{enumerate}[label=(\roman*)]
    \item  Every path not contained in a cycle $C$ of $\mathcal E$ is an element of $I$.
    
    \item If $\rho$ is a linear combination of paths in $e_xKQ_{B}e_y$ such that $q \rho\in I$ (or $\rho q\in I$) for every supplement path $q$ in a cycle $C\in \mathcal{E}$, then $\rho$ is an element of $I$.

    \item If $\rho$ is a linear combination of paths in $e_xKQ_{A}e_y$ such that each summand of $\rho$ is an element of $\langle \C\rangle$ and $\rho q \in I^{\mathcal E}_y$ or $q \rho\in I^{\mathcal E}_x$ for every supplement $q$ in a cycle $C\in \mathcal{E}$, then $\rho$ is an element of $I$.
\end{enumerate}
\end{definition}

\begin{lemma}\label{lem:Split_extension_and_isomorphism}
Let $A$ and $A'$ be Artin algebras such that $T(A)\cong T(A')$. Then $T(A)$ is a split-by-nilpotent extension of $A'$.
\end{lemma}

\begin{proof}
Since $T(A')$ is a split-by-nilpotent extension of $A'$, there exists a short exact sequence of $A'$-$A'$-bimodules
\begin{center}
\begin{tikzcd}
0\ar{r} & DA' \ar{r}{i} & T(A') \ar[shift right, swap]{r}{\pi} & A' \ar{r} \ar[shift right, swap]{l}{\iota} & 0
\end{tikzcd}
\end{center}
such that $\pi\circ \iota= Id_{A'}$.

Now, let $f: T(A')\to T(A)$ be an isomorphism of $K$-algebras. Consider $\pi\circ f^{-1}: T(A)\to A'$ and $f\circ \iota: A' \to T(A)$. Then we have a short exact sequence
\[
\begin{tikzcd}
0\ar{r} & \operatorname{ker}(\pi\circ
f^{-1}) \ar{r} & T(A) \ar[shift right, swap]{r}{\pi\circ f^{-1}} & A' \ar{r} \ar[shift right, swap]{l}{f\circ \iota} & 0
\end{tikzcd}
\]
Moreover, $(\pi\circ f^{-1}) \circ (f \circ \iota)=Id_{A'}$, then $T(A)$ is a split-by-nilpotent extension of $A'$.
\end{proof}

\begin{theorem}\label{thm:characterisation_trivial_extensions}
Let $A=KQ/I$ be an algebra. 
Then $A$ is isomorphic to the trivial extension of some finite-dimensional $K$-algebra if and only if:

\begin{enumerate}
    \item[(a)] There is a presentation of $A$ for which there exists a set $\E$ of distinguished cycles in $Q_A$ with weight function $\omega: \E \to K$ and
    \item[(b)] there is an allowable cut $\C =\{\gamma_1, \cdots, \gamma_t \}$ of $A$ such that the quotient $B=A/\langle \C\rangle$ verifies the following:
\begin{enumerate}
\item [(i)] $A$ is a split-by-nilpotent extension of $B$.
\item [(ii)] The supplements of the cut arrows in the cycles in $\mathcal{E}$ are in one-to-one correspondence with the elements of a basis of $\soc_{B^e} B$.
\end{enumerate}
\end{enumerate}
In this case $A$ is isomorphic to $T(B)$ and the two-sided ideal $\langle \C \rangle$ is isomorphic to $DB$ as a $B$-$B$-bimodule.
\end{theorem}

\begin{proof}
Suppose first that $A$ is isomorphic to $T(B)$.
Then it follows from Theorem~\ref{thm:ideal} that for any $K$-basis $\M$ for $\soc_{B^e}B$, the set of elementary cycles $\E$ induced by $\M$ is a set of distinguished cycles and the set $\{\beta_p: p \in \M\}$ is an allowable cut of $A$ with respect to $\E$. 
Moreover, (b)(i) and (b)(ii) hold by construction.

\bigskip
To prove the converse, we show that $\langle \C \rangle \cong DB$ as $B$-$B$-bimodules and $A\cong T(B)$.
By Lemma~\ref{lem:Split_extension_and_isomorphism} we can assume that the presentation of $A$ is as in the statement.

Let us see first that $\langle \C \rangle\subset DB$. Then by (b)(i) $A$ is a split-by-nilpotent extension of $B$ by $\langle \C \rangle$ and thus $A =  B \ltimes \langle \C \rangle \subset B \ltimes DB = T(B)$. 
Now in order to show $\langle \C \rangle\subset DB$, 
consider a distinguished cycle $C \in \E$.
Up to cyclic permutation, we can write $C= \gamma q$ where $\gamma \in \C$ is an arrow from $x$ to $y$ and $q$ is a path from $y$ to $x$. 
By hypothesis, the path $q$ is in one-to-one correspondence with an element $p$ of $\soc_{B^e} B$ from $y$ to $x$. 
Then, by (b)(ii), $\gamma = p^* \in DB$ and $\langle \C \rangle \subset DB$.  

\medskip
To finish the proof it is enough to show that $DB \subset \langle \C \rangle$.
Let $p$ be an element in the basis of $\soc_{B^e} B$.
Then $p$ induces at least one elementary cycle $C$ of $T(B)$. 
Let $q_p$ be a path from $x$ to $y$ such that $\beta_pq_p = C$, thus $q_p$ is a path in $B$.
Since $q_p$ is a subpath of an elementary cycle in $T(B)$, $q_p$ is non-zero in $T(B)$ and thus non-zero in $B$.
Hence, $q_p$ is non-zero in $A$ and there exists a path $q'$ such that $q_p q' = C'$ is a distinguished cycle of $A$. 

We claim that $q'$ is in fact an arrow. 
From the fact that $q_p$ belongs to $B$, we have that $q'$ contains a unique arrow $\gamma \in \C$ which appears only once in $q'$, i.e., $q' = u \gamma v$ where $u$ and $v$ are paths in $B$.
If $q_p$ is a maximal non-zero path in
$B$ then $q_p$ is an element of the $\soc_{B^e} B$, and we get that $u$ and $v$ are idempotents, therefore $q' = \gamma$. 
It follows that $q_p$ is a supplement of $\gamma$ in $C'$ and then we can identify
$p^*$ with $\gamma$.

Otherwise, $q_p$ is in one-to-one correspondence with an element $p$ in the $\soc_{B^e} B$ of the form $p = a_p q_p + \sum_{i=1}^{t} a_iq_i$ in $\soc_{B^e} B$. 
Suppose that there are two supplements $\alpha$ and $r$ of $q_p$, where $\alpha$ is an arrow.
Then, by Definition~\ref{def:distinguished}, $\alpha-r$ is in $I_A$, which contradicts the assumption that $I_A$ is an admissible ideal.
Hence, if there is a supplement $\alpha$ of $q_p$ which is an arrow, then it is unique and it can be identified with $p^*$.

Now suppose that any supplement $r$ of $q_p$ in $A$ is of the form $r = u \alpha v$ where $\alpha \in \C$  and at least one of $u$ or $v$ is a path of positive length.
Moreover, at least one of the arrows of every $r$ is an arrow of $B$.
Since $A$ is a split-by-nilpotent extension of $B$, we have in particular that $B$ is a subalgebra of $A$.
Moreover $p$ is in $\soc_{B^e} B$, if the length of $u$ is greater than zero, $p (u\gamma v)=(pu)(\gamma v) = 0(\gamma v) = 0$  or, similarly, if the length of $v$ is greater than zero $r p = 0$. 
Hence $rp =0$ or $pr=0$ for every supplement $r$ of $q_p$.
By Definition~\ref{def:distinguished} we get that $p \in I_A$. 
This, together with the fact that $p$ is a linear combination of paths in $B$ implies that $p\in I_B$.
But this is a contradiction of our hypothesis that $p$ is an element of a basis of $\soc_{B^e} B$. 
This contradiction comes from the supposition that $r$ is not an arrow. 
Thus we can conclude that $DB$ is contained in $\langle \C \rangle$.

From the above we conclude that $T(B)$ and $A$ have the same quiver and that the elementary cycles of $T(B)$ coincide with the distinguished cycles of $A$ with the same weight function. 
Thus, it follows from Definition~\ref{def:distinguished} and Theorem~\ref{thm:ideal} that $A$ is isomorphic to $T(B)$ and $\langle\C \rangle$ is isomorphic to $D(B)$ as $B$-$B$-bimodules.
In particular $A$ is a symmetric algebra. 
\end{proof}

In  Theorem~\ref{thm:characterisation_trivial_extensions} we have characterised algebras that are isomorphic to a trivial extension  in terms of allowable cuts. 
We now use this result to study when two algebras have isomorphic trivial extensions in terms of admissible cuts.

\begin{corollary}\label{cor:isomorphic_trivial_extensions}
Let $A= KQ_A/I_A$ be a finite-dimensional algebra with trivial extension $T(A)$.
The following are equivalent
\begin{enumerate}[label=(\alph*)]
\item $T(A)\cong T(A')$
\item There exists a presentation of $T(A)$, and an admissible cut $\C$ of that presentation such that 
\begin{enumerate}
    \item[(i)] $A'\cong T(A)/\langle \C\rangle$,
    \item[(ii)] $T(A)$ is a split-by-nilpotent extension of $A'$,
    \item[(iii)] the supplements of each arrow of $\C$ in the elementary cycles of $T(A)$ are in one-to-one correspondence with the elements of $\soc_{A'^e} A'$.
\end{enumerate}
  
\end{enumerate}
\end{corollary}

\begin{proof}

Suppose that $T(A)\cong T(A')$. Then the quiver of $T(A)$ is equal to the quiver of $T(A')$. 
Hence $T(A')$ is a different presentation of $T(A)$.
Then the result follows from Theorem~\ref{thm:ideal}.
For the converse it suffices to apply Theorem~\ref{thm:characterisation_trivial_extensions} to the algebra $T(A)$.
\end{proof}

In the following example, we illustrate an algorithm to decide whether an algebra is a trivial extension.

\begin{example}
Consider the algebra $A$ given by the quiver $Q_A$
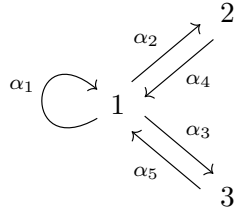
\[
\begin{tikzcd}
  &2 \ar[dl, "\alpha_4", shift left=0.8ex]\\
1 \arrow[out=210,in=140, loop, "\alpha_1"] \ar[ru, "\alpha_2", shift left=0.8ex] \ar[rd, "\alpha_3", shift left=0.8ex]& \\
 & 3 \ar[lu, "\alpha_5", shift left=0.8ex]
\end{tikzcd}\]
with ideal of relations 
$$I_A= \langle \alpha_1^2, \alpha_4\alpha_2, \alpha_5\alpha_3, \alpha_5\alpha_2, \alpha_4\alpha_3, \alpha_4\alpha_1\alpha_3, \alpha_5\alpha_1\alpha_2, \alpha_2\alpha_4-\alpha_3\alpha_5 \rangle.$$
We claim that the cycles $C_1=\alpha_1\alpha_2\alpha_4$ and $C_2=\alpha_1\alpha_3\alpha_5$ and all its permutations are the distinguished cycles of $Q_A$ with weights $w(C_1)=w(C_2)=1$. 
In order to show that the claim holds we need to verify that $I_A$ is determined by $C_1$, $C_2$ and its cyclic permutations, as indicated in  Definition~\ref{def:distinguished}.
Indeed, all the monomial relations correspond to paths that are not subpaths of a distinguished cycle and the relation $\alpha_2\alpha_4-\alpha_3\alpha_5$ is obtained from Definition~\ref{def:distinguished}.(iii).
Moreover the set $\C=\{\alpha_4, \alpha_5\}$ form an allowable cut of $Q_A$ with respect to this set of cycles. 
Finally, the algebra $B = A / \langle \alpha_4, \alpha_5\rangle$ is isomorphic to the path algebra of the quiver $Q_B$
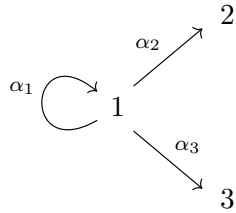
\[
\begin{tikzcd}
  &2 \\
1 \arrow[out=210,in=140, loop, "\alpha_1"] \ar[ru, "\alpha_2", shift left] \ar[rd, "\alpha_3", shift right]& \\
 & 3
\end{tikzcd}\]
modulo $I_B=\langle\alpha^2_1\rangle$.
A basis of the two-sided socle soc$B_{B^e}$ of $B$ is given by the paths $\alpha_1\alpha_2$ and $\alpha_1\alpha_3$.
Hence there is a correspondence between a basis of soc$B_{B^e}$ and the set $\C$.
Finally, it follows from \cite[Theorem~2.5]{ACT} that $A$ is a split-by-nilpotent extension of $B$. 
Then, we can apply Theorem~\ref{thm:characterisation_trivial_extensions} to conclude that $A$ is the trivial extension of $B$.
\end{example}

\section{Wakamatsu's Theorem for bound path algebras}

One of the motivating questions for this paper was to determine the explicit relationship between two finite dimensional algebras that have isomorphic trivial extensions. This has been abstractly described by Wakamatsu in \cite{Wakamatsu1984}, where he also gives necessary and sufficient conditions to decide when two trivial extensions of Artin algebras are isomorphic. 
In this section, we give an explicit description of Wakamatu's result in terms of quivers and relations by providing  an independent proof.

\begin{theorem}\label{thm:Wakamatsu_for_path_algebras}
Let $A= KQ_A/I_A$ be a finite-dimensional algebra with trivial extension $T(A)$ and set $(Q_A)_1=\{\alpha_1, \ldots, \alpha_n\}$ and $\{\beta_1, \dots, \beta_t\}= (Q_{T(A)})_1 \setminus (Q_A)_1$.
The following are equivalent
\begin{enumerate}[label=(\alph*)]
\item $T(A)\cong T(A')$.
\item There exists an admisible cut $\C$ of the form
$$\C\;=\;\{ \alpha_1, \ldots, \alpha_r, \beta_1, \ldots, \beta_s: \alpha_i \in (Q_A)_1 \text{ and } \beta_j\not \in (Q_A)_1\}$$ with $A'\cong T(A)/ \langle \C \rangle$, $T(A)$ is a split-by-nilpotent extension of $A'$ and the supplements in the elementary cycles of the cut arrows are in  one-to-one correspondence with the elements of  $\soc_{A'^e}
A'$.
\item $A \cong S \ltimes M$ and $A' = S \ltimes D(M)$, where 
\begin{enumerate}
    \item[(i)] $S$ is the subalgebra of $T(A)$ generated by $\sum_{x \in (Q_{T_A})_0} e_x$ the identity of $T(A)$ and $\alpha_{r+1}, \ldots, \alpha_n$,
    \item[(ii)] $M = S \left\langle \alpha_1, \ldots, \alpha_r \right\rangle S$,
    \item[(iii)] and $D(M) = S \left\langle \beta_{s+1}, \ldots, \beta_t\right\rangle S$.
\end{enumerate}
\end{enumerate}
\end{theorem}

\begin{proof}
 By Corollary~\ref{cor:isomorphic_trivial_extensions} $(a)$ and $(b)$ are equivalent. 

\medskip
We now show that $(b)$ implies $(c)$. Suppose that $A$ and $A'$ are as in $(b)$ and let $S$ be the subalgebra of $T(A)$, and  $M$ and $D(M)$ be the $S$-$S$-bimodules as defined in $(c)$.

The first step is to show that $A \cong S \ltimes M$.
Let $N \subset T(A)$ be the $S$-$S$-bimodule generated by $\{\beta_{s+1}, \ldots, \beta_t\}$, hence $N^2=0$. 
By construction $M \subset \langle \C \rangle$ and therefore we also have that $M^2=0$. 
We claim that $S \cap (M + N) = 0 $ in $T(A)$.
Suppose that we have $0 \neq s \in  S \cap (M + N)$. 
Then $s = \sum a_i u_i= \sum b_j v_j$ where each $u_i$ is a path from $x$ to $y$ in $S$ and each $v_j$ is a path from $x$ to
$y$ containing at least an arrow in $\{\alpha_1, \ldots, \alpha_r\}$ or in $\{\beta_{s+1}, \ldots, \beta_t\}$. 
If $s$ is a non-zero element of $T(A)$, at least one of the $u_i$ is non-zero, lets call it $u_k$. 
By Theorem~\ref{thm:ideal} there exists a path $q_k \in T(A)$ from $y$ to $x$ such that $q_k u_k$ is an elementary cycle in $T(A)$, that is, $q_k$ is a supplement of
$u_k$.
Recall that each elementary cycle of $T(A)$ contains exactly one arrow of each admissible cut.
Hence the elementary cycle $q_k u_k$ contains exactly one arrow in $\C$ and
one arrow in $\{\beta_1, \ldots, \beta_t\}$.
Since $u_k$ is in $S$, it contains no arrow in $\C \cup \{\beta_1, \ldots, \beta_t\}$. Thus $q_k$  contains exactly one arrow in $\C$ and one arrow in $\{\beta_1, \ldots,
\beta_t\}$.
Furthermore $q_{k} s$ is non-zero and thus there exists $j$ such that $q_k v_j$ is non-zero. 
By definition $v_j$ has at least one arrow of $\C$ or $\{\beta_1, \ldots, \beta_t\}$. 
This implies that $q_{k} v_j$ has two arrows in the same admissible cut. 
This implies that $q_k v_j$ is not a subpath of an elementary cycle and thus $q_k v_j = 0$ by Theorem~\ref{thm:ideal} and we arrive at a contradiction. 
Thus $S \cap (M+N) = 0$.
In particular, we have that $S \cap M = 0 $ and $S \cap N = 0$.

Recall that $S$ is a subalgebra of $T(A)$ generated by  $e_x $ and $\alpha_{r+1}, \ldots, \alpha_n$. By Theorem~\ref{thm:ideal} the relations of $T(A)$ only depend on the relations of
$A$ and the elementary cycles in $T(A)$. 
Thus $S$ can be seen as a  subalgebra of $A$ generated by $e_x$ and $\alpha_{r+1}, \ldots, \alpha_n$.
Now let $a \in A$, then $a = u +v $ where $u \in S$ and $v \in M$. Then for $m \in M$, $am = u m + v m = u m$ since $v m \in M^2 =0$. By
symmetry $M$ is a two-sided ideal of $A$.
Since $S \cap M =0$, the map $h : S \ltimes M \to A$ defined by $h(s,m)= s+m$ is an isomorphism and $S \ltimes M \cong A$.

\medskip
Now, we prove that $A' = S \ltimes DM $.
Since $N$ is a $S$-$S$-bimodule such that $N^2=0$ we get that $\overline{A'} := S + N$ is a subalgebra of $T(A)$.
We have that $\overline{A'}$ is isomorphic to $S \ltimes N$ using a similar argument that we used to show that $A \cong S \ltimes M$.

We now show that $\overline{A'}\cong A'$.
Consider the canonical map $\pi: T(A) \to A'$ restricted to $\overline{A'}$.
Since $\pi|_{\overline{A'}}$ surjects the generators of $\overline{A'}$ onto the generators of $A'$ and the relations of $A'$ are induced by the relations in $T(A)$, we have that $\pi|_{\overline{A'}}$ is a
surjective algebra morphism.

To show that $\pi|_{\overline{A'}}$ is an isomorphism, we need to show that the kernel $\ker(\pi|_{\overline{A'}})=\langle \C \rangle \cap \overline{A'} = \langle \C \rangle \cap (S + N) = 0$.
Let $0 \neq \rho \in \langle \C \rangle \cap (S + N)$.
Then $\rho$ is in $\langle \C \rangle$ and there are $s \in S$ and $n \in N$ such that  $\rho = s + n$, where $ s=\sum a_ku_k $ and $n=\sum b_k v_k$.
Assume $\rho$ is a linear combination of paths from $x$ to $y$.
Note that $\langle \C \rangle \cap N = 0$ since $\langle \C \rangle$ is generated by $\{\beta_1, \dots, \beta_s\}$, $N'$ is the ideal generated by $\{\beta_{s+1}, \dots, \beta_t\}$ and $\{\beta_{1}, \dots,
\beta_t\}$ is an admissible cut of $T(A)$.
If $s$ is non-zero, there exists a $k$ such that $u_k$ is a non-zero path in $T(A)$.
Then there exists a path $q_k$ from $y$ to $x$ such that $q_k u_k$ is an elementary cycle.
Thus $q_k u_k$ contains exactly one arrow of the admissible cut $\C$ and one arrow of $\{\beta_{1}, \dots, \beta_t\}$.
Since $u_k$ is a path in $S$ we have that $q_k$ must contain one arrow of the admissible cut $\C$ and one arrow of $\{\beta_{1}, \dots, \beta_t\}$.
This implies that $q_k v_k =0$ for all $k$ because $N$ is generated by $\{\beta_{s+1}, \dots, \beta_t\}$. 
Thus that $q_k \rho = q_k s \neq 0$.
On the other hand, $\rho \in \langle \C \rangle \cap (S + N)$.
Then we can write $\rho = \sum c_jw_j  \in \langle \C \rangle = \langle \alpha_1, \dots, \alpha_r, \beta_1, \dots, \beta_t \rangle$.
So $w_j$ contains an arrow of $\C$, for all $j$. 
Since we have seen that $q_k$ also contains an arrow of $\C$ we have that $q_k w_j =0$, for all $j$.
Hence $q_k \rho = 0$.
A contradiction that arise from the assumption that $\rho \neq 0$.
Then $(S+N)\cap \langle \C \rangle = 0$, as claimed.

To finish, we need to show that $D(M) \cong N$.
We recall that $N=\langle \beta_{s+1}, \dots, \beta_t \rangle$ where $\beta_k = (0, p_k^*)$ and $p_k^*: A \to K$ is an element of the basis of $D(soc_{A^e} (A))$ which we extend to a basis of $D(A)$.
Thus $N \subset D(A)$.
We define $f : N \to D(M)$ where $f(w): M \to K$ is given by $f(w)(m) = w(m)$ for all $w \in N$ and $m \in M$.

We show that $f: N \to D(M)$ is injective.
Let $0 \neq w \in N$ such that $f(w)(m)=0$ for all $m \in M$.
Without loss of generality, $w = \sum b_k v_k$, where $v_k$ is a path from the vertex $x$ to $y$.
If $w$ is not zero, there exists a $k$ such that $v_k$ is not zero in $T(A)$.
Then there exists a path $q_k$ from $y$ to $x$ such that $q_kv_k$ is an elementary cycle.
By construction, $v_k$ contains exactly one arrow $\{\beta_{s+1}, \dots, \beta_t\}$.
Hence $q_k$ does not contain any arrow in $\{\beta_1, \dots, \beta_t\}$, and $q_k$ is an element of $A$.
On the other hand, the elementary cycle $q_kv_k$ contains exactly one arrow $\alpha$ in $\C$ and $v_k$ does not, thus $q_k$ contains exactly one arrow $\alpha$ of
$\C$.
This implies that $q_k \in M$.
Then $f(w)(q_k)=w(q_k) \neq 0$, which is a contradiction.
So $f$ is injective.

Since $(b)$ implies $(a)$ by Corollary~\ref{cor:isomorphic_trivial_extensions}, $T(A) \cong T(A')$.
Then $\dim_kT(A)=\dim_kT(A')$ and this implies that $\dim_kT(A) = 2 \dim_k A$. 
In particular $\dim_kA = \dim_kA'$. 
From $A = S \ltimes M$ and $A'
= S \ltimes N'$ it follows that $\dim_k DM = \dim_k M = \dim_k N$. 
Then $f$ is surjective and $N$ is isomorphic to $DM$. 
We conclude that $A' \cong S \ltimes DM$.

Now we prove that $(c)$ implies $(a)$.
Since $A = S \ltimes M$ and $A' = S \ltimes DM$ we have that 
$T(A)= A \ltimes DA = S \ltimes M \ltimes (DS \ltimes DM)$ and
$T(A)= A' \ltimes DA' = S \ltimes DM \ltimes (DS \ltimes M)$.
It is not difficult to show that $T(A) \cong T(A').$
\end{proof}

\appendix
\section{The ideal of relations of the trivial extension of an algebra  by Elsa Fernandez}\label{appendix}

In this appendix we give an explicit description of the ideal of relations $I_{T(A)}$ of the trivial extension $T(A) = A \ltimes D(A)$ for any finite-dimensional algebra $A=KQ/I$.

For completeness, we first restate the explicit description of the ordinary quiver of
$T(A)$ based on the quiver of $A$, as given in Section~\ref{sec:background}.  
This description depends on a $K$-basis $\M=\{p_1, \dots, p_n\}$ of $\soc_{A^e}A$. 
The set of vertices $(Q_A)_0$ and $(Q_{T(A)})_0$ coincide.
The set of arrows $(Q_{T(A)})_1$ is the disjoint union of $(Q_A)_1$ and $\{\beta_{p_1}, \dots, \beta_{p_n}\}$ where $s(\beta_{p_i})=t(p_i)$ and $t(\beta_{p_i})=s(p_i)$ for every $p_i \in \M$.

Before giving an explicit description of the ideal of  relations $I_{T(A)}$ of $T(A)$, we recall the following definition from \cite{Fernandez2002}.
For each vertex $x\in Q_{T(A)}$ we define the two-sided ideal $I_x'$  in $KQ_{T(A)}$ generated by

\begin{enumerate}[label=(\roman*)]
\item oriented cycles from $x$ to $x$ which are not contained in an elementary cycle;
\item the elements of the form $\omega (C')C-\omega (C)C'$, where $C$ and $C'$ are elementary cycles starting and ending at $x$.
\end{enumerate}

\begin{theorem}\label{thm:ideal}
Let $A= KQ_A / I_A$ be a finite-dimensional algebra and let $T(A) = KQ_{T(A)}/I_{T(A)}$ be its trivial extension.
Then the quiver $Q_{T(A)}$ is as above and the ideal $I_{T(A)}$ is generated by the union of the following sets.

\begin{enumerate}
    \item A generating set of the ideal of relations $I_A$ of $A$.
    \item The paths that are not contained in an elementary cycle.
    \item For any vertices $x$ and $y$ in $Q_{T(A)}$, the linear combinations of paths $\rho\in e_xKQ_{T(A)}e_y$ such that $q \rho\in I_x'$ or $\rho q\in I_{y}'$ for any supplement path $q$ in an elementary cycle $C$.
\end{enumerate}
\end{theorem}

In order to prove our result we need to recall and show some preliminary results.
We start with the following remark.

\begin{remark}\label{rmk:remark0}
\begin{enumerate}[label=(\roman*)]
    \item Note that it follows directly from Definition~\ref{def:elementary} that each non-zero path $q$ in $KQ_{A}$ can be completed to an elementary cycle in $KQ_{T(A)}$. In
other words, each non-zero path in $KQ_A$ is contained in an elementary cycle $C$.
\item  In Item $3.$ of Theorem~\ref{thm:ideal}, if we set $x=y$, then $\rho$ is a relation generated by the second condition in the definition of $I'_x$ above.
Moreover, every relation of the form $\omega (C')C-\omega (C)C'$ arises in this way by taking $q= e_x$ the trivial path at $x$.

\item Also note that it follows from Item $2.$ of Theorem~\ref{thm:ideal} that if $C= \alpha_1 \dots \alpha_n$ is an elementary cycle in $T(A)$, then the paths $C\alpha_1$ and $\alpha_nC$ are zero in $T(A)$.
\end{enumerate}
\end{remark}

Consider the morphism of $K$-algebras $\Phi: KQ_{T(A)} \to T(A)$ defined on the stationary paths and arrows as follows.
\begin{center}
\begin{tabular}{lll}
  $\Phi(e_x)=(e_x,0)$, &  &  $\text{ for all $x \in (Q_A)_0$}$\\
  $\Phi(\alpha)=(\alpha, 0)$, &  & $\text{ for all $\alpha\in (Q_A)_1$}$,\\
  $\Phi(\beta_p)=(0,p^*)$, & & \text{ and $p\in\mathbb{M}$}.
\end{tabular}
\end{center}

It follows immediately from the definition of $\Phi$ that it is a surjective morphism of algebras and $\ker \Phi \cong I_{T(A)}$. 
Associated with $\Phi$ we consider two morphisms
$$\phi_1 = \pi_1 \Phi: KQ_{T(A)} \to A \qquad \text{and} \qquad \phi_2 = \pi_2 \Phi: KQ_{T(A)} \to D(A)$$
where $\pi_1$ and $\pi_2$ are the natural projections induced by the decomposition of $T(A)=A\oplus D(A)$ as $K$-vector spaces.
We denote by $I_\mathbb{M}=\langle \beta_p \rangle_{p\in \mathbb{M}}$ the ideal of $T(A)$ generated by the elements $\beta_p$ in $T(A)$.
The following lemma was shown in \cite[Lemma 3.5]{Fernandez2002} provided that each cycle in $A$ is zero, but the same proof holds for any algebra $A=KQ/I$.

\begin{lemma}\label{lem:lemma1}
Let $A=KQ_A/I_A$ be an algebra and let $q, u$ be paths in $KQ_{T(A)}$.
\begin{enumerate}[label=(\alph*)]
    \item If $v\in KQ_{T(A)}$ is such that $v=v_1+v_2$ with $v_1\in KQ_A$ and $v_2 \in I_\M$, then $\Phi(v)=(\phi_1(v_1), \phi_2(v_2))$.
    \item $\phi_2(q)\neq 0$ implies $q\in I_\M$.
    \item $\phi_2(q) = 0$ if $q$ contains at least two arrows $\beta_p$, $p \in \M$.
    \item $\phi_2(q)(u)\neq 0$ implies that $u$ is a supplement of $q$.
    \item $\phi_2(v)(u)=\phi_2(vu)(e_y)=\phi(uv)(e_x)$ if $u$ is a path from $x$ to $y$ in $Q_A$.
    \item If $v=\sum a_iq_i$, with $q_i$ different paths and $\phi_2(v)\neq 0$, then there exists a supplement $u$ of the $q_i$ such that $\phi_2(uv)\neq 0$ and
        $\phi_2(vu) \neq 0$.
    \item Let $C$ be an elementary cycle with origin $e$.
    Then $\phi_2(C)(e)=\omega (C)$ and $\phi_2(C)(u)=0$ for any path $u \neq e$.
    \item If $q$ has a supplement, then $\Phi(q) \neq 0$.
\end{enumerate}
\end{lemma}






As a consequence of Lemma~\ref{lem:lemma1}(g) we have that any elementary cycle is nonzero in $T(A)$. 
It was proven in \cite[Corollary 2.8]{Fernandez2002} that if any oriented
cycle in $Q_A$ is zero in A, then any nonzero cycle in $T(A)$ is an elementary cycle. We illustrate with the following example that the previous statement is not true in general, see Example~\ref{ex:nonelementarycycle}.

The following result was shown in \cite[Proposition~3.6]{Fernandez2002} under the assumption that $A$ is monomial. 
However, the proof directly generalises to every finite-dimensional algebra $A$.

\begin{proposition}\label{prop:ideal_of_T(A)}
Let $A$ be a finite-dimensional algebra and $\Phi$ the $K$-algebra morphism defined as above.
Then $I_x'$ is a subset of $\ker \Phi \cap (e_x KQ_{T(A)} e_x$) for any $x\in (Q_A)_0$.
\end{proposition}

\begin{remark}\label{rem:remark1}
As a direct consequence of the preceding proposition we have that the classes of elementary cycles starting at a vertex $x$ of $KQ_{T(A)}$ generate a one
dimensional subspace of $KQ_{T(A)}/I'_x$.
\end{remark}

We now show Theorem~\ref{thm:ideal}

\begin{proof}[Proof of Theorem~\ref{thm:ideal}]
Let $I'$ be the ideal generated by 1.-3. in the statement.
It is sufficient to prove that $I' = \ker \Phi$.

First we show that $I' \subset \ker \Phi$.
It is clear that $I_A \subset \ker \Phi$.
Now, Remark~\ref{rmk:remark0}(i) and Lemma~\ref{lem:lemma1}(d) implies that $\Phi(u)=0$ for every path $u$ which is not contained in an elementary cycle, since if $u$ is contained in an elementary cycle it is non-zero on its supplement. 

Assume now that $w=\sum_{k=1}^l a_k v_k \not\in \ker \Phi$, where $v_k$ are different paths in $I_\M$ from $x$ to $y$.
Then $\Phi(w)=(0,\phi_2(w))$.
We know that $I'_x \subset \ker \Phi$ by Proposition~\ref{prop:ideal_of_T(A)} for all $x\in (Q_{T(A)})_0$.
Then Lemma~\ref{lem:lemma1}.(f) states that if $\phi_2(w) \neq 0$ then there are no supplements $\rho_k$ of $v_k$ such that $\sum_{k=1}^l a_k \rho_kv_k \in I'_x$.
Then $w$ is not in $I'$.
This shows that $I' \subset \ker \Phi$.

Since $\Phi: KQ_{T(A)} \to T(A)$ is surjective and $I' \subset \ker\Phi$, it is enough to show that $\dim_K (KQ_{T(A)}/I')=\dim_K T(A) = 2 \dim_K A$.
Note that the inclusion morphism $\sigma: A \to T(A)$ factors through $KQ_{T(A)}/I'$ because $I_A \subset I'$, which implies that $\iota: A \to KQ_{T(A)}/I'$ is an algebra
monomorphism.

We have that $KQ_{T(A)}= KQ_A + I_\M$.
Then $e_xKQ_{T(A)}e_y= e_x KQ_A e_y + e_xI_\M e_y$ for each $x, y \in (Q_{T(A)})_0$.
Let $\pi: KQ_{T(A)} \to KQ_{T(A)}/I'$ be the canonical epimorphism.
We define  the subspaces $\P_{xy}=\iota(e_x A e_y)$ and $\F_{xy}=\pi(e_xI_\M e_y)$.
Note that $\sum_{xy}\dim_K\P_{xy}=\dim_KA$.

We start by showing that $\P_{xy} \neq 0$ if and only if $\F_{yx}\neq 0$.
Indeed, if $\F_{yx} \neq 0$ then there exists a non-zero path $q \in \F_{yx}$ which admits a supplement $r$ in an elementary cycle $C$ by Remark~\ref{rmk:remark0}(i).
Moreover, $r \in KQ_A$ since $q$ contains an arrow $\beta \in I_\M$ and $C$ is an elementary cycle.
Hence $0 \neq r \in \P_{xy}$ and $\P_{xy}\neq 0$.

Conversely, if $\P_{xy}\neq 0$ there exists a non-zero path $q \in \P_{xy}$ which is contained in an elementary cycle $C$ with supplement $r$.
Moreover, we know that there exists a unique arrow $\beta \in I_\M$ which is contained in the cycle $C$.
Since $C=rq$ and $q \in \P_{xy}$, we have that $r$ contains $\beta$.
Thus $r \in \F_{yx}$ and $\F_{yx}$ is non-zero.

We now prove that $\dim_K\P_{xy} \geq \dim_K\F_{yx}$.
Suppose to the contrary that $n := \dim_K\P_{xy} < \dim_K\F_{yx}$.
Then there is a set of linearly independent paths $\{\mu_1, \dots, \mu_{n+1}\}$ in $\F_{yx}$.
Furthermore, $\mu_{t}$ does not belong to $I'$ for all $1 \leq t \leq n+1$.
This implies that $\mu_t$ is included in an elementary cycle $C_t$ and admits a supplement $\gamma_t$ for all $1 \leq t \leq n+1$.
Note that, $0 \neq \gamma_t\in \P_{xy}$ for all $t$.
Since $n = \dim_K\P_{xy}$, there exists $1\leq s \leq n+1$ such that $\gamma_s = \sum_{i=1}^{s-1} a_i\gamma_i$, where $\{\gamma_1, \dots, \gamma_{s-1}\}$ is a linearly independent set in $\P_{xy}$.
Then there exists $1 \leq r \leq s-1$ such that $\mu_s\gamma_r \neq 0$ since $0 \neq a_s\mu_s\gamma_s= \sum_{i=1}^{s-1} a_i\mu_s\gamma_i$.

Now, since $\mu_s\gamma_r$ is a non-zero path going from the vertex $x$ to itself we have the existence of a cycle $\rho$ from $x$ to $x$ in $KQ_A$ such that
$\rho\mu_s\gamma_r$ is an elementary cycle in $KQ_{A}$.
By hypothesis $A$ is a finite-dimensional algebra, implying the existence of $m \in \mathbb{N}$ such that $\rho^{m-1}\neq 0$ and $\rho^{m}=0$.
Moreover we have that $\{\rho^{m-1}\mu_1, \dots , \rho^{m-1}\mu_{n+1}\}$ is a linearly independent set in $\F_{yx}$.
We know from Remark~\ref{rem:remark1} that the set of elementary cycles from $x$ to $x$ generate a subspace of dimension one in $KQ_{T(A)}/I'_i$.
Thus, there exists a non-zero $b \in K$ such that $\rho\mu_s\gamma_r= b \mu_r\gamma_r$ and then the element $\rho\mu_s-b\mu_r$ belongs to class (3) of $I'$.
Therefore, $\rho\mu_s - b\mu_r=0$.
As a consequence $\mu_r = b^{-1} \rho\mu_s$.
Then $0 \neq \rho^{m-1}\mu_r= b^{-1} \rho^{m}\mu_s = 0$, a contradiction.
So $\dim_K\P_{xy} \geq \dim_K\F_{yx}$ as claimed.
From this we can conclude that $\dim_K T(A)\geq \dim _K KQ_{TA}/{I'}$, since 
\begin{eqnarray*}\dim_K T(A)&=& 2\dim_K A
=\sum_{x,y\in Q_0} \dim_K \P_{xy} + \dim_K
\P_{yx}\\ &\geq& \sum_{x,y\in Q_0} \dim_K \P_{xy} + \dim_K \F_{yx} \geq \dim_k(KQ_{TA}/I'). \end{eqnarray*}

Now, using that $KQ_{TA}/I'$ maps onto $T(A)$, we obtain that $\dim_K(KQ_{TA}/I')\geq \dim_{K}T(A)$. This completes the proof of the theorem.
\end{proof}

\paragraph*{\bf Acknowledgements} 
The authors would like to thank Ibrahim Assem for insightful comments on a previous version of this manuscript. 
EF and ST have been partially supported by a PICT of ANPCyT, Argentina. 
SS and HT were supported by the EPSRC through the Early Career Fellowship, EP/P016294/1.
HT was partially funded by the DFG under Germany’s Excellence Strategy Programme – EXC-2047/1 – 390685813. 
HT is also supported by the European Union’s Horizon 2020 research and innovation programme under the Marie Sklodowska-Curie grant agreement No 893654. 
YV is also supported by the European Union’s Horizon 2020 research and innovation programme under the Marie Sklodowska-Curie grant agreement No H2020-MSCA-IF-2018-838316.
SS acknowledges partial support from the DFG through  the project SFB/TRR 191 Symplectic Structures in Geometry, Algebra and Dynamics (Projektnummer 281071066– TRR 191).
SS, HT and YV would like to thank the Isaac Newton Institute for Mathematical Sciences, Cambridge, for support and hospitality during the programme “Cluster Algebras and Representation Theory” where part of the work on this paper was undertaken. This work was supported by EPSRC grant no EP/K032208/1 and the Simons Foundation.


\begin{thebibliography}{10}

\bibitem{ABScluster-tilted}
I.~Assem, T.~Br\"{u}stle, and R.~Schiffler.
\newblock Cluster-tilted algebras as trivial extensions.
\newblock {\em Bull. Lond. Math. Soc.}, 40(1):151--162, 2008.

\bibitem{ACT}
I.~{Assem}, F.~U. {Coelho}, and S.~{Trepode}.
\newblock {The bound quiver of a split extension.}
\newblock {\em {J. Algebra Appl.}}, 7(4):405--423, 2008.

\bibitem{Assem2013}
I.~Assem, M.~A. Gatica, and R.~Schiffler.
\newblock The higher relation bimodule.
\newblock {\em Algebr. Represent. Theory}, 16(4):979--999, 2013.

\bibitem{Assem2016}
I.~Assem, M.~A. Gatica, R.~Schiffler, and R.~Taillefer.
\newblock Hochschild cohomology of relation extension algebras.
\newblock {\em J. Pure Appl. Algebra}, 220(7):2471--2499, 2016.

\bibitem{BFPPT10}
M.~Barot, E.~Fern\'{a}ndez, M.~I. Platzeck, N.~I. Pratti, and S.~Trepode.
\newblock From iterated tilted algebras to cluster-tilted algebras.
\newblock {\em Adv. Math.}, 223(4):1468--1494, 2010.

\bibitem{Bergh2017}
P.~A. Bergh and D.~O. Madsen.
\newblock Hochschild homology and trivial extension algebras.
\newblock {\em Proc. Am. Math. Soc.}, 145(4):1475--1480, 2017.

\bibitem{Fernandez2002}
E.~A. Fern\'{a}ndez and M.~I. Platzeck.
\newblock Presentations of trivial extensions of finite dimensional algebras
  and a theorem of {S}heila {B}renner.
\newblock {\em J. Algebra}, 249(2):326--344, 2002.

\bibitem{FernandezTesis}
E. Fernández
\newblock Extensiones triviales y álgebras inclinadas iteradas.
\newblock PhD thesis, Universidad Nacional del Sur, Argentina, 1999.

\bibitem{Fernandez2006}
E. Fern\'{a}ndez and M.~I. Platzeck.
\newblock Isomorphic trivial extensions of finite dimensional algebras.
\newblock {\em J. Pure Appl. Algebra}, 204(1):9--20, 2006.

\bibitem{Grant2019}
J.~Grant.
\newblock Higher zigzag algebras.
\newblock {\em Doc. Math.}, 24:749--814, 2019.

\bibitem{GR76}
E.~L. Green and I.~Reiten.
\newblock On the construction of ring extensions.
\newblock {\em Glasgow Math. J.}, 17(1):1--11, 1976.

\bibitem{Green}
E.~L. Green and S.~Schroll.
\newblock Almost gentle algebras and their trivial extensions.
\newblock {\em Proc. Edinb. Math. Soc. (2)}, 62(2):489--504, 2019.

\bibitem{HernandezTesis}
M.~V. {Hernandez}
\newblock {Á}lgebras autoinyectivas y extensiones triviales de álgebras monomiales
\newblock Master Thesis.
\newblock Universidad Nacional del Sur, 2017.

\bibitem{Hochschild1945}
G.~Hochschild.
\newblock On the cohomology groups of an associative algebra.
\newblock {\em Ann. Math. (2)}, 46:58--67, 1945.

\bibitem{HW83}
D.~Hughes and J.~Waschb\"{u}sch.
\newblock Trivial extensions of tilted algebras.
\newblock {\em Proc. London Math. Soc. (3)}, 46(2):347--364, 1983.

\bibitem{IW79}
Y.~Iwanaga and T.~Wakamatsu.
\newblock Trivial extension of {A}rtin algebras.
\newblock In {\em Representation theory, {II} ({P}roc. {S}econd {I}nternat.
  {C}onf., {C}arleton {U}niv., {O}ttawa, {O}nt., 1979)}, volume 832 of {\em
  Lecture Notes in Math.}, pages 295--301. Springer, Berlin, 1980.

\bibitem{MU74}
W.~M\"{u}ller.
\newblock Unzerlegbare {M}oduln \"{u}ber artinschen {R}ingen.
\newblock {\em Math. Z.}, 137:197--226, 1974.

\bibitem{Nagata}
M.~Nagata
\newblock Local rings. 
\newblock \emph{Interscience Tracts in Pure and Applied Mathematics 13}. New York and London: Interscience Publishers, a division of John Wiley and Sons. XIII, 234 p. (1962).


\bibitem{Schroll2015}
S.~Schroll.
\newblock Trivial extensions of gentle algebras and {B}rauer graph algebras.
\newblock {\em J. Algebra}, 444:183--200, 2015.

\bibitem{Wakamatsu1984}
T.~Wakamatsu.
\newblock Note on trivial extensions of {A}rtin algebras.
\newblock {\em Comm. Algebra}, 12(1-2):33--41, 1984.

\end{thebibliography}
\end{document}